\documentclass[11pt]{article}
\usepackage{amsfonts}
\usepackage[width=17cm, left=1cm]{geometry}
\usepackage{amssymb}
\usepackage{amsmath}
\usepackage{amsthm}
\usepackage{tikz}
\usetikzlibrary{graphs}

\newtheorem{definition}{Definition}
\newtheorem{proposition}{Proposition}
\newtheorem{theorem}{Theorem}
\newtheorem{remark}{\noindent Remark}
\newtheorem{lemma}{Lemma}

\newtheorem{problem}{Problem}

\title{Classification of measurable functions of several variables and matrix distributions}

\author{A.~M.~Vershik \thanks{St.~Petersburg Department of Steklov Institute of Mathematics and St.~Petersburg State University, St.~Petersburg, Russia; Institute for Information Transmission Problems, Moscow, Russia. Supported by the ISF grant
21-11-00152.}}
\date{}

\begin{document}

\maketitle
\begin{flushright}\emph{To the memory of my correspondence teacher I.~M.~Gelfand}\end{flushright}


\bigskip
\begin{abstract}
We consider the notion of the matrix (tensor) distribution of a measurable function of several variables. On the one hand, it is an invariant of this function with respect to a certain group of transformations of variables; on the other hand, there is a special probability measure in the space of matrices (tensors) that is invariant under actions of natural infinite permutation groups. The intricate interplay of both interpretations of matrix (tensor) distributions makes them an important subject of modern functional analysis. We state and prove a theorem saying that, under certain conditions on a measurable function of two variables, its matrix distribution is a~complete invariant.

\end{abstract}

\section{Introduction, statement of the problem, overview of basic notions}

\subsection{General remarks}

The matrix distribution of a function $f$ of $k$ variables is a measure on the set of tensors of rank~$k$ (which are matrices if $k=2$), or, in other words, a random matrix or tensor. For $k=1$, it is simply a~Bernoulli measure on the space of sequences of (independent) values of $f$.
For $k=2,3 \dots$, this random matrix (tensor) is constructed from a multidimensional Bernoulli sample of values.

Here is a simplest example for  $k=2$: if $f$ is a metric on a space $X$ with continuous measure  (a~metric triple, see below and
\cite{VVZ}), then its matrix distribution is a measure on the restrictions of $f$ to countable Bernoulli grids of points in $X$, i.e., a random metric on a countable Bernoulli subset. As we will see, in the general case, for the matrix distribution of a function to be a complete invariant of this function with respect to permutations of variables, we should endow it with an additional structure.

The notion of matrix distribution first appeared, for various reasons, in works by Aldous~\cite{Al1}, Gromov~\cite{Gr} (in the case of metrics), and the author~\cite{VUM}  (see below). Later, it was studied in a~series of papers: \cite{V02, VH,VH1, VVZ}. We refine and extend the previous research and pose several questions related to applications of the suggested techniques.

\subsection{Definition of isomorphism of measurable functions}

The space with continuous measure on which we define a measurable function $f$ is usually the same up to isomorphism, namely, it is  the standard  probability space (Lebesgue space). However, in this space one can define various $\sigma$-subalgebras, their products, and identifications. Thus, the notion of an individual variable is given a precise interpretation as a parameter associated to some $\sigma$-subalgebra of the $\sigma$\nobreakdash-algebra in the domain of $f$. In the principal case, the only case we are considering, the measure space where $f$ is defined is equipped with a structure of a {\it direct product of  standard $\sigma$-algebras}, and each subalgebra has its own measure and determines its own variable. Thus, the domain of $f$ is represented as the direct product  $(X_1 \times X_2 \times \ldots \times X_k, \mu_1 \times \mu_2 \times \ldots \times \mu_k)$  of two or finitely many standard Lebesgue spaces with continuous measures. The $\sigma$-algebra of all measurable sets in this space is a direct product of $\sigma$-algebras: 
$${\frak A}={\frak A}_1\times {\frak A}_2\times \dots \times{\frak A}_k.$$

Note that here the variables are independent in the sense of measure theory; generalizations to more complicated cases (where variables that are independent in the common sense are not independent in the sense of probability theory) can also be conveniently interpreted in terms of decomposing the global $\sigma$-algebra  into various configurations of $\sigma$-subalgebras. Here, we are discussing the simplest case of a~direct product. It is worth mentioning that, in the author's view, the theory of finite collections of $\sigma$-subalgebras of the standard space and their invariants should become a central subject in geometric measure theory.

Let us introduce a notion of isomorphism of measurable functions that is natural for this category.

\begin{definition}
Two measurable functions $f_i(\cdot,\cdot,\dots,\cdot)$, $i=1,2$,  of  $k$ variables, defined on products of spaces
$\prod_{j=1}^k  (X_j^i, \mu_j^i)$, $i=1,2$, respectively, are said to be isomorphic if there exist invertible measurable transformations
$T_j:X_j^1 \longrightarrow X_j^2$, $j=1,\dots k$,  such that $T_j\mu_j^1=\mu_j^2$ and
$$ f_2(T_1 x_1,\dots ,T_k x_k)=f_1(x_1,x_2,\dots x_k)$$ for almost all collections of variables.
\end{definition}

This definition can be further specified for particular cases. For instance, if there is a fixed isomorphism between the $\sigma$-algebras ${\frak A}_j$, then, setting $T_j = T$ for all $j$ and considering symmetric functions, we define the notion of isomorphism for symmetric functions of several variables. {\it We pose the problem of finding the invariants of measurable functions with respect to the notion of isomorphism introduced above}. Essentially, we are discussing the action of the group of measure-preserving transformations (or a product of such groups) on the space of $\bmod 0$ classes of  measurable functions, and the {\it orbit space} of this action. A description of this space is suggested below.

\subsection{One or several variables}

In the case of a single variable, we are talking about isomorphism between two measurable functions, $f_1,f_2$, defined on the interval $[0,1]$: $f_2(Tx)=f_1(x)$ for a.e.\ $x$, where $T$ is a $\bmod 0$ invertible measurable mapping from the interval to itself preserving the Lebesgue measure.
This problem was solved long  ago in terms of Rokhlin's theory of measurable partitions (see \cite{Ro}). First of all, it is obvious that for measurable $\bmod 0$ one-to-one  functions $f:X\rightarrow \Bbb R$, the distribution of $f$   in the usual case (the measure $f_*(\mu)$ on the image) is a complete invariant of $f$; while the space of Borel measures on the image is precisely the orbit space of the isomorphism problem for one-to-one functions.

In the case of non-one-to-one functions, there arises a measurable partition into the preimages of points (with multiplicities), to which Rokhlin's theorem on the general form of invariants of a~measurable partition can be applied. Thus, the classification problem for functions of one variable is completely solved, and both a complete invariant and the space of complete invariants are described.

It may seem that the subsequent discussion suggests that the problem can also be completely solved in the case of several independent variables. However,  classification of functions of two or more variables requires entirely new ideas. The answers (invariants and their spaces) may be quite different. To tackle the classification problem for functions $f$ of several variables, we suggest to study the restrictions of $f$ to random countable subsets (grids) of a certain type  (instead of the distribution of~$f$ itself or of simplest transformations of $f$, as in the one-dimensional case). The probability distributions of these restrictions allow us, for highly nontrivial reasons, to obtain complete invariants. Essentially, our method consists in approximating functions defined on continuous sets with functions on countable subsets endowed with a specific (Bernoulli or Bernoulli-like) structure. This reminds the {\it Monte Carlo approach} to computations. Interestingly, even in the one-dimensional case, the justification of this method relies on the pointwise ergodic theorem, according to which the distribution of a~function coincides with the limit of its empirical distributions. However, in the multidimensional case the situation is more complicated, at least because there is no equivalent of the distribution of~$f$; it is precisely the matrix distribution of $f$ that is designed to replace it.

 \subsection{Grids in measure spaces and purity of functions}

We begin with measurable functions $f$ of two variables defined on a measure space. We may assume without loss of generality that $f$ is defined on the square  $[0,1]^2$ endowed with the Lebesgue measure~$\mu^2$.

\begin{definition}
A function $f$ (of $k=2$ variables) is said to be pure if the following condition is satisfied: if we fix almost every value of one variable, say $x$ (or $y$), then the restrictions  $y \rightarrow f(\cdot,y)$ (resp.,  $x \rightarrow f(x, \cdot)$)  of $f$ 
to the other variable
are $\bmod 0$ one-to-one mappings.
\end{definition}

Clearly, this definition generalizes the notion of a one-to-one function of one variable. 

It is natural to extend this definition to a function $f$ of $k$ variables with $k>2$ by requiring the same condition of being one-to-one to be satisfied if we fix all variables except one; in this case, we also say that $f$ is {\it pure}. But if the number of variables $k$ is greater than~$2$, then, fixing almost every collection of some
$k-r$ variables, $r=1,2,\dots, k-1$, we may require the restrictions to the remaining $r$ variables to be  $\bmod 0$ one-to-one; if this condition is satisfied, we say that the original function of $k$ variables is $r$-pure. However, we will restrict ourselves to the notion of $1$-pure functions (which we simply call pure).

Our main theorem will apply to pure functions of $k$  variables with $k>1$.  The extension of the obtained results to general functions of  $k>1$ variables is more or less the same as the analogous extension to non-one-to-one functions of one variable in Rokhlin's theorem. In our exposition, we are dealing with pure functions of two variables, because all the reasoning remains valid for pure functions with a greater number of variables.
The  {\it purification} of a function $f(\cdot,\cdot)$ of two variables defined on
$X_1\times X_2$ is the function $\bar f$ on the space $X_1/\xi_1\times X_2/\xi_2$ where $\xi_i$ is the partition of  $X_i$ into the maximal classes of points with the same restrictions of $f$ to the other variable. Clearly, a pure function coincides with its purification.

Consider a pure function $f$ of two variables defined on a product
$X=(X_1 \times X_2, \mu^1\times \mu^2)$ of Lebesgue spaces. We consider all functions to be real just for convenience; in fact, the range may be any standard Borel space.

We will identify the factors $X_1, X_2$ with the interval $[0,1]$, and assume that both measures coincide with the Lebesgue measure.
The construction of isomorphism invariants of such a function $f$  is based on the study of the restrictions of  $f$ to certain random countable subsets, traditionally referred to as {\it grids}.\footnote {Although their role is different from the role of grids in the theory of computation.} Thus, a grid is an element of the space $(X\times X)^{\infty}=X^{\infty}\times X^{\infty}$.
In the definition, we assume that the space $X=X_1\times X_2$ is defined explicitly (it is obvious that everything is well defined with respect to identification $\bmod 0$).

We will need two types of grids. The first one is Bernoulli grids defined as follows. Choose two (in general, different) sequences of independent identically distributed points, $\{x_n\}_n$ and $\{y_m\}_m$. The collection of points $\{x_n,y_m\}_{(n,m)}$ will be called a {\it Bernoulli grid}.
If $f$ is symmetric and we are studying invariants of symmetric functions (for instance, if $f$ is a metric), then the sequences coincide and the grid is symmetric.

It is more convenient to consider two-sided grids $\{x_n\}_n, \{y_m\}_m$, i.e., to assume that 
$n,m$ run over the group
$\Bbb Z$; however, everything  remains valid even in the case of the semigroup ${\Bbb Z}_+$.

Obviously, every such sequence is dense in $X_1$ (or $X_2$), while the grid is dense in  $X=X_1\times X_2$; it is also clear that any two Bernoulli grids in $X$ are isomorphic.

The second type of grids, which we will discuss later, is locally finite grids: these are increasing sequences of finite subsets
 $H_n=H^1_n\times H^2_n \subset X=X_1 \times X_2$,  where  $\{H^i_n\}_{n \in \Bbb N}$ for $i=1,2$ are uniformly distributed sequences on $X_i$, $i=1,2$. In both cases, we are interested in asymptotic properties of restrictions of measurable functions to grids.

\section{Matrix distributions and classification of functions}

\subsection{Definition}
Let $f(\cdot,\cdot)$ be a pure function of two variables on $X=X_1\times X_2$ and
$(\{x_n\}_n, \{y_m\}_m)$ be a fixed Bernoulli grid. We regard the two-dimensional array $${||f(x_n, y_m)||_{(n,m)\in {\Bbb Z}^2}}$$
as an infinite (with four infinities) real matrix. Thus, we introduce a mapping 
$$F_f: X^{\infty}\times X^{\infty} \longrightarrow \operatorname{Mat}_{\infty \times \infty} ({\Bbb R}), $$
where
     $$F_f (\{x_n\}_n\times \{y_m\}_m)=||f(x_n, y_m)||_{(n,m)\in {\Bbb Z}^2}.$$

Now we can regard $F_f$ as a mapping from the Borel space of Bernoulli grids in the measure space
$(X^\infty,\mu^{\infty})\times (X^\infty,\mu^{\infty})$ to the space $\operatorname{Mat}_{\infty \times \infty} ({\Bbb R})$ of infinite matrices. It preserves the action of the group of translations on the image.

Since in what follows we will use the pointwise ergodic theorem, we will assume that all measurable functions under consideration are integrable. The extension of the results to arbitrary measurable functions can be obtained using various artificial techniques, but we will not delve into this here. 

In the general theory of Borel mappings, there is a notion of the mapping of Borel measures associated with a Borel mapping of spaces (see, e.g., \cite{Bo}); the definition involves the idea of full preimages. In accordance with this notion, we can define the image of a Bernoulli measure.

  \begin{definition}
For a pure function $f$, consider the image of the Bernoulli measure $\mu^{\infty}\times \mu^{\infty}$ under the mapping $F_f$:
  $$D_f=(F_f)_*(\mu^{\infty}\times \mu^{\infty}).$$
  The measure $D_f$ on the set of infinite matrices will be called the matrix distribution of~$f$.
 \end{definition}

Recall that this definition originated in the theory of classification of functions  in \cite{V02} and was discussed in \cite{VH, VVZ}. However, it seems that this object first appeared in  \cite{Al1} as a measure invariant under the group of permutations; also, it essentially appeared in \cite{Gr} as an invariant of metrics in measure spaces  (see below).

It clearly follows from the definition that the matrix distribution is a measure in the space of matrices invariant and ergodic with respect to the action of the group $S_{\Bbb Z}\times S_{\Bbb Z}$ of row and column permutations (in the symmetric case, of  the group  $\operatorname{DIAG}(S_{\Bbb Z}\times S_{\Bbb Z})$ of simultaneous row and column permutations).

The following lemma is obvious.

  \begin{lemma}
The matrix distribution (regarded as a measure on the space of matrices) is an invariant of a measurable function with respect to isomorphism (of measurable functions of two variables).
  \end{lemma}

  \begin{proof}
If two functions, $f_1$ and $f_2$, are isomorphic, then it follows from the definitions that there exists an isomorphism between the spaces of matrices that maps the matrix distribution of $f_1$ to the matrix distribution of $f_2$. Here we use the fact that all Bernoulli grids are isomorphic to each other. Note that any isomorphism between the spaces of matrices, by definition, commutes with the group of row and column permutations.
  \end{proof}

\subsection{The recovery theorem and the completeness of invariants}

In the statement of the lemma, the function is not assumed to be pure. Obviously, the matrix distribution of a function coincides with the matrix distribution of its purification. However, the main challenge  lies in proving the {\it completeness of the suggested invariant}. It turns out that this problem for the matrix distribution, even for pure functions, is more delicate: it does not suffice to regard the matrix distribution as a measure on a subset in the space of matrices, {\it it is essential to preserve another action on this subset, induced by the group of translations of the grid ${\Bbb Z}^2$} in both variables, which acts in the space of matrices by translations of rows and columns. The point is that Bernoulli grids   $\{x_n\}\times \{y_m\}$ in our case are invariant with respect to the two-sided translation in both variables, i.e., the group ${\Bbb Z}^2$ (or with respect to the one-sided translation in both variables if we are dealing with the semigroup  ${\Bbb Z}^2_+$).

Therefore, the matrix distribution should be regarded as a measure on the space of matrices acted upon by this group. However, the standard definition of the image of a measure does not take into account an additional group action, so we must consider it separately. We will prove that two pure functions are isomorphic if their matrix distributions coincide and the action of  ${\Bbb Z}^2$
  on them is preserved. Without this additional condition, we cannot claim that
  the matrix distribution is a~complete invariant. However, in many (if fact, in almost all) cases, this caveat is not necessary, because the above-mentioned actions coincide automatically,  for various reasons.

To prove this, we need a corollary from the pointwise ergodic theorem for ergodic actions of the group 
$\mathbb Z^k$, which is crucial in problems of this kind. However, the author has not encountered this corollary in the literature on ergodic theory (see \cite{W}). {\it It allows one to uniquely recover an action of a~group on a measure space from its action on one (typical) orbit and some structure on that orbit.}

\begin{lemma}
Consider the group $G={\Bbb Z}^k$ and an invariant ergodic measure $\mu$ on $l^{\infty}(G)$. Let $x_0 \in l^{\infty}(G)$ be a typical point for $\mu$, i.e., a point whose orbit $O$ (regarded as a group) satisfies the following property: there is a distinguished countable collection of sequences of sets in $O$ along which the averages of the chosen functions exist. Then the functional 
$$E(\phi)=\lim_n |Q_n|^{-1} \sum_{g\in Q_n}\phi(gx_0)\equiv \bar \phi$$ 
is defined, and it uniquely determines $\mu$ as the unique measure on the space $X$ of functions on $G$ for which
 $$\int_X \phi(g_0 x)d\mu= \bar \phi.$$ 
Here ${Q_n}$ is a sequence of F{\o}lner sets, for example, the sequence of centrally symmetric cubes with side~$n$ in  ${\Bbb Z}^k$, and $|Q_n|$ is the number of points in $Q_n$.
\end{lemma}

\begin{proof}
The existence of the limits is a direct consequence of the pointwise ergodic theorem, and the uniqueness of a measure with given limits follows from the fact that the specified set of functions is total and, therefore, the measure is uniquely determined if the integrals of all these functions are fixed.\end{proof}

\begin{remark}{\rm 
Slightly deviating from the main topic, we note that the above lemma can be viewed as a tool for studying the isomorphism problem for actions of amenable measure-preserving groups. The functional $E$ introduced above is linear on some space $\Phi_G$ of bounded functions on $G$ (whose image is total in all $L^1(G)$ spaces with respect to translation-invariant ergodic measures). The set of values of~$E$ is a metric invariant of the action (for a fixed choice of the space $\Phi_G$), and, since $E$ can be viewed as an invariant mean on the group, this opens up the possibility of studying the correspondence between certain classes of invariant means on the group and the measure types of actions with  invariant measure. In a broader context, such lemmas allow one to recover continuous objects from countable subobjects. Note that the relationship between invariant means on amenable groups and pointwise ergodic theorems for group actions is not sufficiently studied.}
\end{remark}

The above lemma implies the following.

\begin{lemma}
Let $f$ be a pure integrable function of two variables on $[0,1]^2$. The matrix distribution of~$f$, regarded as a measure on the set of matrices acted upon by the group  ${\Bbb Z}^2$ of translations leaving~$f$ invariant, is a complete system of invariants of~$f$.
\end{lemma}

\begin{proof}
To prove the lemma, it suffices to use the following characteristic property of pure functions (see \cite{V02,VH}). Let
$f\in L^1$ be a pure integrable function of two variables and $\{x_n,y_m\}_{(n,m)}$ be two sequences of i.i.d.\ points in
$(X_1,\mu)$ and $(X_2,\mu)$. Then for almost all such pairs of sequences, the collection of functions
$$ \phi_{x_n}(\cdot)=f(x_n,\cdot), \quad \psi_{y_m}(\cdot)=f(\cdot,y_m) $$ 
is a total set in $L^1(X_1\times X_2, \mu\times \mu)$.
\end{proof}

As proved in \cite{VH1}, the group of measure-preserving transformations leaving any measurable function invariant is compact. Therefore, the quotient by this group is well defined. However, we do not use this fact, since taking the quotient  by the group of symmetries may break the structure of a direct product of $\sigma$-algebras.

The symmetries of the function $f$ are defined by the condition $f(Tx, Sy)=f(x,y)$, where $S,T$ are invertible measure-preserving transformations. If the condition is satisfied only for $S=\operatorname{Id}$, $T=\operatorname{Id}$, then we say that {\it the function $f$ has the trivial (zero) group of symmetries}. In this case, the caveat in the statement of the lemma on the group of translations is not necessary, and we obtain the following theorem.

\begin{theorem}[recovery theorem]
The matrix distribution of a pure integrable function with the trivial symmetry group, regarded as a measure in the space of matrices, is 
a complete system of isomorphism invariants. In more detail, if two integrable pure functions with the trivial symmetry group have the same (i.e., isomorphic) matrix distributions and equal averages in the respective variables, then they are isomorphic.
\end{theorem}

The nontriviality of the theorem lies in the fact that the challenging task of verifying the isomorphism of functions is reduced to  verifying some relations (equalities of integrals) that are relatively easy to check. The proof of the theorem ultimately relies on the pointwise ergodic theorem.

It should be emphasized that under the conditions of the theorem, the matrix distribution, regarded as a measure in the space of matrices, inherits the structure of a matrix space. However, in contrast to \cite{V02, VH}, we do not construct or use a universal model of a measurable function,  even though its construction as a Radon--Nikodym density on an infinite product is straightforward.

For the classification of symmetric functions (see below), the statements should be slightly modified.

\subsection{A necessary example}

The nontriviality of the recovery theorem becomes apparent even when considering simplest scenarios.

Consider the invariants of the following function of two variables ($k=2$) on the unit square $[0,1]^2$: 
$$f(x,y)=x+y \quad \bmod 1.$$
Correspondingly, consider two sequences, $\{(x_n)\}$, $\{(y_m)\}$, of independent points; the space of pairs $ \{(x_n,y_m)\}_{n,m \in \Bbb Z}$ (a grid on $[0,1]^2$); and the group of translations on the infinite product with the Bernoulli measure $(1/2,1/2)$ on $\{0,1\}$.
The mapping $F_f$ is as follows:
$$\{(x_n, y_m)\}_{n,m}\rightarrow \{(x_n+y_m)\}_{n,m}.$$

In this case, Theorem 1 says that $F_f$ is  $\bmod 0$ invertible, and that one can recover each summand (i.e., the function) given their sum (i.e., the matrix distribution). But a set of measure~$1$ on which the mapping $F_f$ is invertible can be described explicitly (!):
$$ \{(x_n,y_m): \lim_{r \rightarrow \infty} r^{-1}\sum_{k=1}^r(x_n,y_k)=(x_n,0), \quad \lim_r r^{-1}\sum_{k=1}^r(x_k,y_m)=(0,y_m)\}.$$

The fact that the measure of this set is equal to $1$ follows from the law of large numbers (ergodic theorem). Clearly, for all $n,m$ and a given sum $x_n + y_m$, the summands are uniquely recoverable.

As we see, in this case the mapping $F_f$ is invertible, meaning it is an isomorphism $\bmod 0$. Interestingly, this simple example demonstrates how, using the ergodic theorem, we can extract a~subset of ``right-hand sides'' of certain equations for which the system of equations $\{x+y=C\}$ is uniquely solvable. In this case, the construction is very straightforward, and there is an inversion formula. However, for an arbitrary function, we only obtain a theorem about the existence of such a set, and there is no explicit construction of the inverse mapping.

\begin{remark}{\rm
One should remember that all matrix distributions under consideration, regarded as subsets in the space of matrices, are equipped with a family of conditional limit distributions of all rows and columns; the claim``two matrix distributions coincide'' implies the coincidence of these limits. The specificity of matrix distributions as measures on the space of matrices lies in the fact that the usual description of measures via finite-dimensional distributions is inefficient, although this description is given explicitly (see the next section about Aldous's theory). It  does not directly imply neither the existence nor a description of the conditional limit distributions of rows and columns for almost all matrices. Apparently, this is the novelty of matrix distributions as measures in spaces of infinite matrices.}
\end{remark}

\subsection{Relation to Aldous's theory}

Recall that Aldous (see \cite{Al1,Al2,Kal}) generalized the classical De Finetti theorem for Bernoulli measures on sequences to the case of matrices by proving the following theorem (we present it in slightly different terms compared to the original ones): any ergodic measure on the set of infinite real matrices that is invariant with respect to any of the following groups:
\begin{itemize}
\item the group $S_{\Bbb N}\times S_{\Bbb N}$  of all permutations of rows and columns,

\item the group $\operatorname{DIAG}(S_{\Bbb N}\times S_{\Bbb N})$ of coinciding permutations of rows and columns,
\end{itemize}
is a probability Borel measure on the space of matrices $||f(\xi_i, \eta_j, \lambda_{i,j}) ||_{i,j}$ determined by a measurable function $f$ of three variables, each ranging over the interval $[0,1]$ and evaluated at the nodes of a~Bernoulli grid (in the first case); in the second case, we have the additional condition that $\eta_j \equiv \xi_j$. Here, $\xi_i, \eta_j,\lambda_{i,j}$ are independent sequences of independent random variables uniformly distributed on the interval $[0,1]$.
 
 We are interested in the special case of this result without  $\lambda_{i,j}$; its connection to the classification problem is as follows.

\begin{proposition}
Every Aldous measure with zero $\lambda$ is a matrix distribution of a measurable function of two variables and it uniquely determines the measure; the function can be assumed to be pure. In the second case, we mean the matrix distribution of a symmetric function (see the next section).
\end{proposition}

Here, we are not discussing Aldous's result in  full generality, or even in the above-mentioned special case  (that with $\lambda=0$), but it should be emphasized that his result demonstrates that the  invariant measures on matrices we are interested in can be described in terms of measurable functions. In contrast, we want to use measures on matrices as invariants of functions and describe the properties of these measures. Apparently, such a description has not been provided so far. A more detailed analysis, including an attempt to prove Aldous's theorem using the author's ergodic method and to establish a connection between both problems, will be undertaken elsewhere. In Aldous's theorem, the case of $\lambda = 0$ (which was overlooked in \cite{74}) is degenerate, but it is the most interesting one precisely because it leads to measures (matrix distributions) whose properties differ  from those of well-studied random matrices with independent elements and the corresponding random spectrum (see below).

\section{The symmetric matrix distribution is a complete invariant of metric triples}

\subsection{Definition}

The classification of symmetric measurable functions of several variables, including metrics in measure spaces, is somewhat different from the general classification of arbitrary measurable functions in measure spaces.
For a detailed description of the joint axiomatics of spaces endowed with measures and metrics, see \cite{Gr,VVZ} and the references therein. We mean so-called metric triples (space, measure, metric) $(X, \mu, \rho)$, where $(X,\rho)$ is a Polish space equipped with a Borel continuous measure $\mu$ that is fully supported with respect to the metric $\rho$. The problem of finding a complete system of invariants under the group of measure-preserving isometries for such triples was posed and solved by Gromov in~\cite{Gr}. A~later proof of this result, suggested by the author, used other (ergodic) methods, in particular, those related to the classification of functions of several variables in Lebesgue spaces. Both proofs are presented in \cite{Gr} and the author's papers \cite{VUM}, see also \cite{VVZ}.

The definition of the matrix distribution of a function given above, which uses arbitrary two-dimensional grids, can be extended to metrics, thereby providing a complete invariant of metrics in the class of general functions of two variables. However, this definition  changes in a natural way when we modify the notion of isomorphism. In particular, when considering symmetric functions of two or more variables (e.g., metrics on measure spaces), it is natural to introduce special grids in order to consider the invariants of symmetric functions in $k$ variables under one and the same transformation of all variables:
$$f_2(Tx_1,\dots, Tx_k)=f_1(x_1,\dots, x_k).$$  The definition of the matrix distribution of a function in this situation (for $k=2$) should naturally be based on grids of the form $\{x_n,x_m\}_{n,m}$.

\begin{definition}
The matrix distribution of a symmetric function $f$ of two variables is a measure in the space of infinite matrices defined by the formula
  $$D_{\rho}=(F_f)_*(\mu^{\infty}\times \mu^{\infty}),$$
 that is, the image  of $\mu^{\infty}\times \mu^{\infty}$ under the following mapping:
 $$F_f (\{x_n\}_n\times \{x_m\}_m)=||f(x_n, x_m)||_{(n,m)\in {\Bbb Z}^2}.$$
In particular, the matrix distribution of a metric, regarded as a measurable function $\rho$ on a measure space $(X, \mu)$, is a measure on the space of {\it distance matrices}, i.e., metrics on a countable space ($\mathbb{Z}$ for two-sided grids and $\mathbb{Z}_+$ for one-sided grids).
\end{definition}

Note that not every ergodic measure on the set of distance matrices that is invariant under simultaneous permutations
is a matrix distribution of a metric triple. For example, for the Bernoulli measure on symmetric $0-1$ matrices, which defines a random (universal) graph, the triangle inequality is automatically satisfied, and this measure is not a matrix distribution of any metric triple, because, in particular, the entropy condition (see below)  is not met.

\subsection{$mm$-entropy of metric triples}

For every metric measure space, the notion of $mm$-entropy is defined, which is the following function of $\epsilon$:
$$ H(\epsilon)=\min \big\{r\in \Bbb N: \mu\big(\bigcup_{i=1}^r V_i(\epsilon, x_i)\big)>1-\epsilon\big\}, $$
where  $V(\epsilon, x)$ is the ball of radius $\epsilon$ centered at $x$. Clearly, $H(\epsilon)$ is finite for finite $\epsilon$, and
we are interested in the germ of this function at zero, which is an important characteristic of the metric triple. This definition apparently has several authors, and one of them is the author of this paper, who used the notion of $mm$-entropy to define the so-called scaling entropy of automorphisms and catalytic invariants, which generalize the Shannon--Kolmogorov entropy of dynamical systems (see \cite{VM,VVZ}). However, it turned out that Shannon also defined the $mm$-entropy of metrics with measure in his famous paper~\cite{Sh} on information theory, although his definition went mostly unnoticed for a long time
(see~\cite{VM}). 
It is not difficult to understand that the $mm$-entropy of a metric triple can be computed from its matrix distribution as a limit of certain functionals of its finite-dimensional fragments. Furthermore, as shown in \cite{VVZ}, the finiteness of the limit is a sufficient condition for an ergodic invariant measure on distance matrices to be a matrix distribution of a metric.

\subsection{Comparing the proofs of the Gromov--Vershik recovery theorem}

Gromov \cite{Gr} proves,   in slightly different terms, the following recovery theorem:
{\it For a metric triple $(X, \mu, \rho)$, the measure on the set of infinite distance matrices that is the unique weak limit of the measures on $n$-dimensional distance matrices corresponding to random collections of $n$ independent identically distributed (according to $\mu$) points in $X$ is a complete set of invariants of $(X, \mu, \rho)$ with respect to the group of $\mu$-preserving isometries.}

The most challenging part is to prove the completeness, i.e., the fact that a metric triple can be uniquely recovered from a given weak limit of random distance matrices. The weak limit under consideration can be called the matrix distribution with respect to the limit over finite grids. Of course, the proof  requires estimates such as the moment method used in \cite{Gr}.

Another approach involves considering, from the very beginning, infinite {\it Bernoulli grids} and their distance matrices. The trick is that almost any (typical) grid allows one to recover the metric triple, because it is immediately clear that the metric can be recovered from a metric on a dense subset. Afterward, the ergodicity is only needed to recover the measure. However, as we have seen,  the recovery based on a single orbit is possible without using the properties of the metric itself, due to deeper ``individual'' reasons. The ergodic theorem plays an essential role, and Lemma 3, which follows from the pointwise ergodic theorem, explains why recovery is possible in the general case without appealing to specific properties of the function.

A question that remains is whether the weak limit of measures emerging from Gromov's ``finite'' approach coincides with the matrix distribution defined above via infinite Bernoulli grids. Without going into details, we assume that this is the case and pose a more general question:

\begin{problem} 
For a given measurable function (e.g., a metric on a measure space),
do the matrix distributions, regarded as measures in the space of matrices, constructed from {\rm 1)} Bernoulli grids, {\rm 2)}~locally finite grids, and {\rm 3)} grids constructed from stationary sequences satisfying the zero--one law, coincide?
\end{problem}

\subsection{Application of the matrix distribution of metrics}

The mapping that sends a metric triple to its matrix distribution can be used to study various properties of metrics. The Lipschitz continuity of this mapping is proved in $\cite{GK}$. Convergence (in the Hausdorff--Gromov sense) of metric measure spaces  can be expressed in terms of convergence of matrix distributions. Unfortunately, there are few examples of explicit computations. It is important to extend the definition of matrix distributions to spaces with poorer structures than that of a metric triple. A typical example: assume that for a sequence of metric triples, the sequence of the corresponding matrix distributions converges, as a sequence of measures in the space of matrices, to some measure that is not a matrix distribution. What can we say about the limits of these triples? The problem is to describe such limits.

\subsection{The spectrum of a matrix distribution}

In conclusion, we briefly touch on an issue important for applications: the spectrum of matrix distributions.
Consider the matrix distribution of a metric triple in the sense of Section~3.1. This is a~measure on distance matrices, i.e., symmetric nonnegative matrices satisfying the triangle inequality. We can ask about the behavior of the spectra of finite fragments of the matrix distribution, treating them as random vectors composed of (real) eigenvalues. Their asymptotic properties are of special interest.

If the matrix distribution is a complete invariant of metrics, it is natural to ask how complete the asymptotic spectral properties can be. This  analog of the famous M.~Kac's question, ``Can one hear the shape of a drum?'', for spectra of metric triples was posed by the author  in  \cite{VUM}. There are some observations about the spectra of metrics, as well as interesting experiments conducted on the author's  initiative in \cite{BBS}. However, there is no clear picture yet. It is important to investigate the asymptotic behavior of spectra not  for the distance matrices themselves, but for matrix distributions, i.e., the asymptotic behavior of metric triples.

It is natural to assume that the metric is always integrable, but not square-integrable with respect to the measure. In the case where it is  square-integrable, the limiting spectrum of the matrix distribution is always deterministic, and it is determined by the spectrum of an integral operator whose kernel coincides with the metric (see \cite{Kol}). Apparently, the first known example of a metric triple with a random limit spectrum is given in \cite{VP}; it is the Euclidean metric on the half-line with the Cauchy measure.

An important question is how close to the semicircle law the limit spectrum of matrix distributions of metric spaces can be. For example, {\it what are the limiting spectra of the metric on the Urysohn universal space} with respect to probability measures on this space?


\begin{thebibliography}{99}

\bibitem{Al1}D.~Aldous, Representations of partially exchangeable arrays of random variables, J. Multivariate Anal., 11 (1981), 381--398.

\bibitem{Al2}D.~Aldous, Exchangeability and related topics, Lecture Notes in Math., 1117 (1985), 1--198.

\bibitem{Bo}V.~I.~Bogachev, Measure theory, Vol. I, II. Springer-Verlag, Berlin, 2007. Vol. I: xviii+500 pp., Vol. II: xiv+575 pp.

\bibitem{BBS} E.~Bogomolny, O.~Bohigas, and  C.~Schmit, Spectral properties of distance matrices,
J. Phys. A, 36 (2003), No.~12, 3595--3616.

\bibitem{GK} S.~Gadgil and M.~Krishnapur, Lipschitz correspondence between metric measure spaces
and random distance matrices, Int. Math. Res. Not. IMRN, 2013, No.~24, 5623--5644.

\bibitem{Gr}M.~Gromov, Metric Structures for Riemannian and Non-Riemannian Spaces,  Progr. Math., 152, Birkh{\"a}user Boston, Boston, MA, 1999, xx+585 pp.

\bibitem{Kal} O.~Kallenberg, On the representation theorem for exchangeable arrays, J. Multivariate Anal., 30 (1989), 137--154.

\bibitem{Kol}V.~Koltchinskii and E.~Gine, Random matrix approximation of spectra of integral operators, Bernoulli, 6 (2000), No.~1, 113--167.

\bibitem{OM}V.~A.~Rokhlin, On the fundamental ideas of measure theory, Mat. Sb. (N.S.), 25(67) (1949), No.~1, 107--150.

\bibitem{Ro}V.~A.~Rokhlin, Metric classification of measurable functions, Uspehi Mat. Nauk (N.S.), 12(1957), No.~2, 169--174.

\bibitem{Sh}C.~E.~Shannon,
A mathematical theory of communication, Bell System Tech. J., 27 (1948),
379--423, 623--656.

\bibitem{74}A.~M.~Vershik, Description of invariant measures for the actions of some infinite-dimensional groups, Sov. Math. Dokl., 15 (1974), 1396--1400.

\bibitem{V02}A.~M.~Vershik, Classification of measurable functions of several arguments, and invariantly distributed random matrices, 
Funct. Anal. Appl., 36 (2002), No~2, 93--105.

\bibitem{VUM} A.~M.~Vershik,	Random metric spaces and universality, Russian Math. Surveys, 59 (2004), No.~2, 259--295.

\bibitem{VM}A.~Vershik, Dynamics of metrics in measure spaces and their asymptotic invariants,
Markov Process. Related Fields, 16 (2010), No.~1, 169--184.

\bibitem{VH1}A.~Vershik and U.~Haboek, Compactness of the congruence group of measurable functions in several variables,
J. Math. Sci. (N.Y.), 141 (2007), No.~6, 1601--1607.

\bibitem{VH}A.~Vershik and U.~Haboek, On the classification problem of measurable functions in several variables and on matrix distributions, J. Math. Sci. (N.Y.), 219 (2016), No.~5, 683--699.

\bibitem{VL}A.~M.~Vershik and M.~A.~Lifshits, On $mm$-entropy of a Banach space with a Gaussian measure, Teor. Veroyatnost. i Primenen., 68 (2023), No.~3, 532--543.

\bibitem{VP}A.~M.~Vershik and F.~P.~Petrov, Limit spectral measures of the matrix distributions of the metric triples, Funktsional. Anal. i Prilozhen., 57 (2023), No.~2, 106--110.

\bibitem{VVZ}A.~M.~Vershik, G.~A.~Veprev, and P.~B.~Zatitskii, Dynamics of metrics in measure spaces and scaling entropy, Uspekhi Mat. Nauk, 78 (2023), No.~3(471), 53--114.

\bibitem{W} B.~Weiss, Single Orbit Dynamics, Amer. Math. Soc., Providence, RI, 2000, x+113 pp.
\end{thebibliography}
\end{document}